\documentclass[a4paper,11pt] {amsart}
\usepackage{amsthm,amssymb,latexsym,mathrsfs}%,  wasysym} %txfonts,
\usepackage{color}
\input amssym.def
 \newtheorem{thm}{Theorem}[section]
 \newtheorem{cor}[thm]{Corollary}
 \newtheorem{lem}[thm]{Lemma}
 \newtheorem{prop}[thm]{Proposition}
 \theoremstyle{definition}
 \newtheorem{defn}[thm]{Definition}
 \theoremstyle{remark}

 \numberwithin{equation}{section}

\newcommand{\bprop} {\begin{proposition}}
\newcommand{\eprop} {\end{proposition}}
\newcommand{\btheo} {\begin{theorem}}
\newcommand{\etheo} {\end{theorem}}
\newcommand{\blem} {\begin{lemma}}
\newcommand{\elem} {\end{lemma}}
\newcommand{\bcor} {\begin{corollary}}
\newcommand{\ecor} {\end{corollary}}

\newcommand{\Be}{\begin{equation}}
\newcommand{\Ee}{\end{equation}}
\newcommand{\Bea}{\begin{eqnarray}}
\newcommand{\Eea}{\end{eqnarray}}
\newcommand{\Bes}{\begin{equation*}}
\newcommand{\Ees}{\end{equation*}}
\newcommand{\Beas}{\begin{eqnarray*}}
\newcommand{\Eeas}{\end{eqnarray*}}
\newcommand{\Ba}{\begin{array}}
\newcommand{\Ea}{\end{array}}

\def\C{\mathbb{C}}

\begin{document}

%-------------------------------------------------------------------------
% editorial commands: to be inserted by the editorial office
%
%\firstpage{1} \volume{228} \Copyrightyear{2004} \DOI{003-0001}
%\seriesextraline{This is the Concrete Title of this Book\br H.E. R and S.T.C. W, Eds.}
%\firstpage{1}

\title[HANKEL OPERATORS ON HARDY-ORLICZ SPACES]{HANKEL OPERATORS ON HOLO\-MORPHIC HARDY-ORLICZ SPACES}

%----------Author 1
\author{Beno\^it F. Sehba}
\address{Centre de Math\'ematiques et Informatique (CMI), Universit\'e de Provence, Technop\^ole Ch\^ateau-Gombert
39, rue F. Joliot Curie,13453 Marseille Cedex 13 France}
\email{bsehba@cmi.univ-mrs.fr}

\thanks{The first author was partially supported by the ANR project ANR-09-BLAN-0058-01.
The second author was supported by the Centre of Recerca Matem\`atica, Barcelona (Spain).}
%----------Author 2
\author{Edgar Tchoundja}
\address{D\'epartement de Math\'ematiques\\ Facult\'e des Sciences\\ Universit\'e de Yaound\'e I\\
B.P. 812\\ Yaound\'e Cameroun}
\email{tchoundjaedgar@yahoo.fr}


\date{May 03, 2012}

\begin{abstract}
We characterize the symbols of Hankel operators that extend into bounded operators from the Hardy-Orlicz $\mathcal H^{\Phi_1}(\mathbb
 B^n)$ into $\mathcal H^{\Phi_2}(\mathbb  B^n)$ in the unit ball of $\mathbb C^n$, in the case where the growth functions $\Phi_1$ and $\Phi_2$ are  either concave or convex. The case where the growth functions are both concave has been studied by Bonami and Sehba.
 %The case  where the growth functions are both convex is also treated here.
 We also obtain  several weak factorization theorems for functions in  $\mathcal H^{\Phi}(\mathbb B^n)$,
 with concave growth function, in terms of products of Hardy-Orlicz functions with convex growth functions.
\end{abstract}

%%% ----------------------------------------------------------------------
\maketitle

\section{Introduction and statement of results}

Let $\mathbb B^n=\{z\in \C^n:|z|<1\}$ be the unit ball of
$\C^n$($n> 1$). We denote by $d\nu$ the Lebesgue measure on
$\mathbb B^n$ and $d\sigma$  the normalized measure on
$\mathbb S^n=\partial{\mathbb B^n}$ the boundary of
$\mathbb B^n$. By $\mathcal H(\mathbb B^n)$, we denote the
space of holomorphic functions on $\mathbb B^n.$

For $z=(z_1,\cdots,z_n)$ and $w=(w_1,\cdots,w_n)$ in
$\C^n$, we let
$$\langle z,w\rangle =z_1\overline {w_1} +
\cdots + z_n\overline {w_n}$$
 so that $|z|^2=\langle
z,z\rangle =|z_1|^2 +\cdots +|z_n|^2$.

  We say a function $\Phi$ is a growth function, if $\Phi$ is a continuous and non-decreasing function from $[0,\infty)$ onto itself.
  We say that $\Phi$ is of lower type if we can find $p>0$ and $C>0$ such that, for $s>0$ and $0<t\le 1$,
\begin{equation}\label{lowertype}
 \Phi(st)\le Ct^p\Phi(s).\end{equation}
 We say that $\Phi$ is of upper type  if we can find $q > 0$ and $C>0$ such that, for $s>0$ and $t\ge 1$,
\begin{equation}\label{uppertype}
 \Phi(st)\le Ct^q\Phi(s).\end{equation}
We say that $\Phi$ is of lower type $p$ (resp. upper type $q$) when (\ref{lowertype}) (resp. (\ref{uppertype})) is satisfied.
 Also, we say that $\Phi$ satisfies the $\Delta_2 -$ condition if there exists a constant $K>1$ such that, for any $t\ge 0$,
\begin{equation}\label{delta2condition}
 \Phi(2t)\le K\Phi(t).\end{equation}
Observe the equivalence between the properties (\ref{uppertype}) and (\ref{delta2condition}).

For $\Phi$ a growth function, we denote by $\mathcal H^\Phi(\mathbb B^n)$ the Hardy-Orlicz space consisting of holomorphic function $f$ in the unit ball $\mathbb B^n$ such that if the functions $f_r$ are defined by $f_r(z)=f(rz)$ then
 $$||f||^{lux}_{\mathcal H^\Phi}:=\sup_{r<1}||f_r||_{L^\Phi}^{lux}<\infty $$
 where $$||f||_{L^\Phi}^{lux}:=\inf \left\{\lambda>0:\int_{\mathbb S^n}\Phi\left(\frac{|f(\xi)|}{\lambda}\right)d\sigma(\xi)\leq 1\right\}$$
 is the Luxemboug (quasi)-norm of $f$ in the Orlicz space $L^\Phi(\mathbb S^n)$.
We will also often consider the following (quasi)-norm on $\mathcal H^\Phi(\mathbb B^n)$, namely
$$||f||_{\mathcal H^\Phi}:=\sup_{r<1}\int_{\mathbb S^n}\Phi(|f(r\xi)|)d\sigma(\xi)$$
which is finite for $f\in \mathcal H^\Phi(\mathbb B^n)$.
 For $0<p<\infty$, when $\Phi(t)=t^p$, the above space corresponds to the usual Hardy space
  $\mathcal H^p(\mathbb B^n)$, that  is
 the space of all $f\in \mathcal H(\mathbb B^n)$ such that $$||f||_{p}^{p} := \sup_{0<r<1}\int_{\mathbb S^n}|f(r\xi)|^{p}d\sigma(\xi) <
 \infty.$$

Two growth functions $\Phi_1$ and $\Phi_2$ are said equivalent if there exists some constant $c$ such that
$$c\Phi_1(ct)\leq \Phi_2(t)\leq c^{-1}\Phi_1(c^{-1}t)\,\,\,\textrm{for all}\,\,\,t>0.$$
Such equivalent growth functions define the same Orlicz space.
We denote by $\mathcal H^\infty(\mathbb B^n)$, the space of bounded holomorphic functions  in $\mathbb B^n$.

The following is proved in \cite{VT}:
\begin{prop}\label{volbergtolokonnikov}
For $\Phi_1$ and $\Phi_2$ two growth functions, the bilinear map $(f,g)\mapsto fg$ sends $L^{\Phi_1}\times L^{\Phi_2}$ onto $L^{\Phi}$, with the inverse mappings of $\Phi_1,\Phi_2$ and $\Phi$ related by
\begin{equation}\label{volbergproduct}
\Phi^{-1}=\Phi_1^{-1}\cdot \Phi_2^{-1}.
\end{equation}
Moreover, there exists some constant $c$ such that
$$||fg||^{lux}_{L^{\Phi}}\le c ||f||^{lux}_{L^{\Phi_1}}||g||^{lux}_{L^{\Phi_2}}.$$
\end{prop}

Let us define two classes of growth functions of our interest in this paper.
\begin{defn}
We call $\mathscr{L}_p$ the set of growth functions $\Phi$ of lower type $p$,  $(0<p\le 1)$, such that the function $t\mapsto \frac{\Phi(t)}{t}$ is non-increasing.
\end{defn}
\begin{defn}
We call $\mathscr{U}^q$ the set of growth functions $\Phi$ of upper type $q$,  $(q\ge 1)$, such that the function $t\mapsto \frac{\Phi(t)}{t}$ is non-decreasing.
\end{defn}
Clearly, functions in $\mathscr{L}_p$ and $\mathscr{U}^q$ satisfy the $\Delta_2-$ condition. Note that if $\Phi \in \mathscr{U}^q$, then $\Phi$ is of lower type $1$. For $\Phi\in \mathscr{L}_p$  (resp. $\mathscr{U}^q$), without loss of generality, possibly replacing $\Phi$ by the equivalent growth function $\int_0^t\frac{\Phi(s)}{s}ds$, we can assume that $\Phi$ is concave (resp. convex) and $\Phi$ is a $\mathcal{C}^1$ function with derivative $\Phi^\prime(t) \simeq \frac{\Phi(t)}{t}.$

For any $\xi\in
 \mathbb S^n$ and $\delta > 0$, let $$ B_{\delta}(\xi)=\{w\in \mathbb S^n : |1-\langle w,\xi\rangle|<
 \delta\}.$$

 We call  a weight
$\varrho$, any continuous increasing function from $[0,\infty)$ onto itself, which is of upper type $\alpha$ on $[0,1]$, that is,
\begin{equation}
\varrho(st)\leq s^\alpha\varrho(t)
\end{equation}
for $s>1,$ with $st\leq 1.$ Given a weight $\varrho,$ we define the space $BMO(\varrho)$ as the subspace of $L^2(\mathbb S^n)$ consisting of
those $f\in L^2(\mathbb S^n)$ such that
\Be\label{bmoaweight} \sup_B\inf_{R\in \mathcal P_N(B) }\frac 1{(\varrho(\sigma(B)))^2\sigma(B)}\int_B|f-R|^2d\sigma=C<\infty,\Ee
where, for $B=B_\delta(\xi_0)$, the space $\mathcal P_N(B)$ is the space of polynomials of
 order $\leq N$ in the  $(2n-1)$ last coordinates related to an orthonormal basis whose first element is $\xi_0$ and
 second element $\Im \xi_0$. Here $N$ is taken larger than $2n\alpha-1$. We set $\|f\|_{BMO(\varrho)}:=\|f\|_{2}+C$, where $C$ is given in the definition (\ref{bmoaweight}) of $BMO(\varrho)$.  The space $BMOA (\varrho)$ is then the space of function $f\in \mathcal H^2(\mathbb B^n)$ such that
 $$\sup_{r<1} \|f_r\|_{BMO(\varrho)}<\infty.$$

 Clearly, $BMOA(\varrho)$ coincides with the space of holomorphic functions in $\mathcal H^2(\mathbb B^n)$
 such that their boundary values lie in $BMO(\varrho)$. When $\varrho=1$, $BMOA(\varrho)$ is the usual space of holomorphic functions with bounded mean
 oscillation $BMOA$.

As pointed out in \cite{BG,BS}, from Viviani's results \cite{V}, $BMOA(\rho)$ spaces appear as duals of particular Hardy-Orlicz spaces.
\begin{thm}\label{duality1}
Let $\Phi\in \mathscr{L}_p$. The topological dual space $(\mathcal{H}^\Phi(\mathbb
 B^n))^*$ of $\mathcal{H}^\Phi(\mathbb
 B^n)$ can be identified with the space $BMOA(\rho)$ (with equivalent norms) under the integral pairing
$$\left\langle f,g \right\rangle = \lim_{r\rightarrow 1}\int_{\mathbb{S}^n}f(r\xi)\overline{g(r\xi)}d\sigma(\xi),$$
when $\Phi$ and $\rho$ are related by
\begin{equation}\label{relationrhoandphi}
\rho(t):=\rho_\Phi(t)= \frac{1}{t\Phi^{-1}(1/t)}.
\end{equation}
\end{thm}
In order to give the dual of $\mathcal{H}^\Phi(\mathbb
 B^n)$ when $\Phi\in \mathscr{U}^q,$ we need to recall the notion of complementary function of a growth function. For $\Phi$ a growth function, the complementary function, $\Psi : \mathbb R_+ \rightarrow \mathbb R_+$, is defined by

\begin{equation}\label{complementarydefinition}
\Psi(s)=\sup_{t\in\mathbb R_+}\{ts - \Phi(t)\}.
\end{equation}
We may verify that if $\Phi\in \mathscr{U}^q$, then $\Psi$ is also a growth function of lower type such that $t\mapsto \frac{\Psi(t)}{t}$ is non-decreasing but which may not  satisfy the $\Delta_2-$conditon. The fact that $\Psi$  also satisfies the $\Delta_2-$conditon is relevant in our results here. We thus introduce another class of growth functions.
\begin{defn}
We say that a growth function $\Phi$ satisfies the $\bigtriangledown_2-$condition whenever  its complementary satisfies the $\Delta_2-$ conditon.
\end{defn}

Several characterizations that guarantee that a growth function has a complementary function satisfying the $\Delta_2-$ condition are known.
%There are several characterizations of growth function $\Phi$ such that its complementary function $\Psi$ satisfies the $\Delta_2-$conditon.
One of these characterizations is the Dini condition which we recall here.
We say that  $\Phi\in \mathscr{U}^q$ satisfies the Dini condition if there exists a  constant $C>0$ such that, for $t>0$,
\begin{equation}\label{dinicondition}
\int_0^t\frac{\Phi(s)}{s^2}ds\leq C\frac{\Phi(t)}{t}.
\end{equation}
So when $\Phi$ satisfies (\ref{dinicondition}), then $\Phi$ satisfies the $\bigtriangledown_2-$condition.
From the duality result in \cite{RR}, we obtain the following result.

\begin{thm}\label{duality2}
Let $\Phi\in \mathscr{U}^q$ and $\Psi$ its complementary function. Suppose that $\Phi$ satisfies the Dini condition (\ref{dinicondition}). Then the topological dual space $(\mathcal{H}^\Phi(\mathbb
 B^n))^*$ of $\mathcal{H}^\Phi(\mathbb
 B^n)$ can be identified with  $\mathcal{H}^{\Psi}(\mathbb
 B^n)$ (with equivalent norms) under the integral pairing
$$\left\langle f,g \right\rangle = \lim_{r \rightarrow 1}\int_{\mathbb{S}^n}f(r\xi)\overline{g(r\xi)}d\sigma(\xi).$$
\end{thm}

 The orthogonal projection of $L^2(\partial \mathbb
 B^n)$ onto $\mathcal H^2(\mathbb B^n)$ is called the
 Szeg\"o projection and denoted $P$. It is given by
 \begin{equation} P(f)(z)=\int_{\partial \mathbb B^n}S(z,\xi)f(\xi)d\sigma(\xi),\end{equation}
 where $S(z,\xi)=\frac{1}{(1-\langle z,\xi \rangle)^n}$ is
 the Szeg\"o kernel on $\partial \mathbb B^n$. We denote as
 well by $P$ its extension to $L^1(\partial \mathbb
 B^n)$.

 For $b\in \mathcal H^2(\mathbb B^n)$, the small Hankel
 operator with symbol $b$ is defined for $f$ a bounded
 holomorphic function by $h_b(f):=P(b\overline f)$.

 In this paper we are interested in the
 boundedness of the small Hankel operators $h_b$ from $\mathcal H^{\Phi_1}(\mathbb
 B^n)$  to $\mathcal H^{\Phi_2}(\mathbb
 B^n)$.

 In the one dimensional case, that is the unit disc of the complex plane $\mathbb C$,  boundedness of the small Hankel operator
 between Hardy spaces has been considered in \cite{JPS} and completely solved in
 \cite{T}. A. Bonami and S. Madan in \cite{BM} used the so-called ``balayage" of Carleson measures to characterize
 symbols of bounded Hankel operators between Hardy-Orlicz spaces in the unit disc of $\mathbb C$.
It is well known that $h_b$ extends as a bounded
 operator on $\mathcal H^p(\mathbb B^n)$ for $p>1$ if and
 only if $b$ is in $BMOA$ (see \cite{CRW}).

 Recently, some of the one dimensional results have been extended to the unit ball $\mathbb B^n$. First,
 using some simple techniques, A. Bonami, S. Grellier and the first author proved in \cite{BGS} that $h_b$ is bounded on $\mathcal H^1(\mathbb
 B^n)$ if and only if $b\in BMOA(\varrho)$ with $\varrho(t)=\frac{1}{\log\left(\frac{4}{t}\right)}$.  In \cite{BG}, A. Bonami and S. Grellier using weak factorization results
 were able to characterize symbols of bounded Hankel operators from the space $\mathcal H^\Phi(\mathbb
 B^n)$  to $\mathcal H^1(\mathbb
 B^n)$, where $\Phi\in \mathscr{L}_p$. The two last works have been extended in \cite{BS} to the case of Hankel operators between two Hardy-Orlicz spaces $\mathcal H^{\Phi_1}(\mathbb
 B^n)$ and $\mathcal H^{\Phi_2}(\mathbb
 B^n)$ with $\Phi_i \in \mathscr{L}_p;$ $i=1,2.$

 Since in the papers \cite{BG} and \cite{BS} the Orlicz functions were assumed to be concave, there are many interesting cases in which the question of the boundedness of $h_b$ is still open. In \cite{ST}, the authors provided a characterization of bounded Hankel operators, $h_b$, from $\mathcal H^{\Phi_{p}}(\mathbb
 B^n)$ to $\mathcal H^q(\mathbb  B^n)$, in terms of the belonging of the symbols to some weighted Lipschtiz spaces,  here $\Phi_{p}=\left(\frac{t}{\log(e+t)}\right)^p$, $p\le 1,$ $0<q<\infty $.   We remark that $\Phi_{p}\in \mathscr{L}_p$ and  for $q>1$, $\Phi_2(t)=t^q$ is in $\mathscr{U}^q$.

 In this paper we consider the boundedness of $h_b$ between the Hardy-Orlicz spaces $\mathcal H^{\Phi_1}(\mathbb
 B^n)$ and $\mathcal H^{\Phi_2}(\mathbb
 B^n)$, where $\Phi_1$ and $\Phi_2$ are either in $\mathscr{L}_p$ or $\mathscr{U}^q$ but not both in $\mathscr{L}_p$. When the functions $\Phi_2\in \mathscr{U}^q$, we restrict to those satisfying
\begin{equation}\label{extracondition}
\lim_{x\rightarrow\infty}\frac{\Phi_2(x)}{x}=\infty.
\end{equation}
In fact, if $\Phi_2\in \mathscr{U}^q$ does not satisfy (\ref{extracondition}) then $\Phi_2$ is equivalent to $\Phi(x)=x$ so that  $\mathcal H^{\Phi_2}(\mathbb B^n)=\mathcal H^{1}(\mathbb B^n)$. This case has been settled in \cite{BG} and note that (\ref{extracondition}) implies that $q>1$.

The simple and direct approach used in \cite{ST} seems difficult to be used
here for this general  situation. We will be inspired instead by the techniques and methods in \cite{BG,BS}. The main tool is the use of the molecular decomposition of Hardy-Orlicz spaces given in \cite{BG} to obtain our needed weak factorization. The fact that molecules in the molecular decomposition in \cite{BG} can have arbitrary large order will be crucial.
In the next section we will prove the following results.

\begin{thm}\label{maintheorem1}
Let $\Phi_1\in \mathscr{L}_p$ and $\Phi_2\in \mathscr{U}^q$, $\rho_i(t)=\frac{1}{t\Phi_i^{-1}(1/t)}$ and $\Psi_2$ the complementary of $\Phi_2.$ Then the product of two functions, one in $\mathcal H^{\Phi_1}(\mathbb
 B^n)$ and the other one in $\mathcal H^{\Psi_2}(\mathbb
 B^n)$, is in $\mathcal H^{\Phi}(\mathbb
 B^n)$, with $\Phi$ such that
\begin{equation}
\rho_{\Phi}:=\frac{\rho_1}{\rho_2},
\end{equation}
or, equivalently,
\begin{equation}
\Phi^{-1}(t):=\Phi_1^{-1}(t)\Psi_2^{-1}(t).
\end{equation}
Moreover, functions in   $\mathcal H^{\Phi}(\mathbb
 B^n)$ can be weakly factorized in terms of products of functions of
$\mathcal H^{\Phi_1}(\mathbb
 B^n)$ and $\mathcal H^{\Psi_2}(\mathbb
 B^n)$.
\end{thm}
\begin{thm}\label{maintheorem2}
Let $\Phi_1\in \mathscr{L}_p$ and $\Phi_2\in \mathscr{U}^q$, $\rho_i(t)=\frac{1}{t\Phi_i^{-1}(1/t)}$ and assume that $\Phi_2$ satisfies the Dini condition (\ref{dinicondition}). Then   the Hankel operator $h_b$ extends into a bounded operator from $\mathcal H^{\Phi_1}(\mathbb
 B^n)$ into $\mathcal H^{\Phi_2}(\mathbb
 B^n)$ if and only if its symbol $b$ belongs to  $ BMOA(\rho_\Phi)= (\mathcal H^{\Phi}(\mathbb
 B^n))^*,$ where
\begin{equation*}
\rho_{\Phi}:=\frac{\rho_1}{\rho_2}.
\end{equation*}
\end{thm}

In section \ref{section2}, we study the boundedness of the Hankel operators $h_b$ from $\mathcal H^{\Phi_1}(\mathbb
 B^n)$ into $\mathcal H^{\Phi_2}(\mathbb
 B^n)$, when $\Phi_i\in \mathscr{U}^q,\;i=1,2.$ To deal with this situation, because of the convexity of both growth functions, we will have to  rewrite in a slightly general form, the molecular decomposition in \cite{BG}. This will allow us to obtain other weak factorization results for functions in $\mathcal H^{\Phi}(\mathbb B^n)$, with $\Phi\in \mathscr{L}_p,$ in terms of products of functions  of $\mathcal H^{\Phi_1}(\mathbb
 B^n)$ and $\mathcal H^{\Phi_2}(\mathbb
 B^n)$, with $\Phi_i\in \mathscr{U}^q,\;i=1,2.$ This generalizes the classical result in \cite{CRW,GP}. We obtain in this situation the following result.

\begin{thm}%\label{maintheorem3}
Let $\Phi_1$ and $\Phi_2$ in $ \mathscr{U}^q$, and $\rho_i(t)=\frac{1}{t\Phi_i^{-1}(1/t)}$. \\
We suppose that:
\begin{itemize}
\item[(i)] $\Phi_2$ satisfies the Dini condition (\ref{dinicondition})
\item[(ii)]  $\frac{\Phi_1^{-1}(t)\Psi_2^{-1}(t)}{t}$ is non-decreasing or $\Phi_1=\Phi_2.$
\end{itemize}
 Then   the Hankel operator $h_b$ extends into a bounded operator from $\mathcal H^{\Phi_1}(\mathbb
 B^n)$ into $\mathcal H^{\Phi_2}(\mathbb
 B^n)$ if and only if its symbol $b$ belongs to  $ BMOA(\rho_\Phi),$ where
\begin{equation*}
\rho=\rho_{\Phi}:=\frac{\rho_1}{\rho_2}.
\end{equation*}
\end{thm}

 Finally, throughout the paper, $C$ will be a constant not necessarily the same at each occurrence. We will also use the notation $C(k)$
 to express the fact that the constant depends on the underlined parameter. Given two positive quantities $A$ and $B$, the notation
 $A\lesssim B$ means that $A\le CB$ for some positive constant $C$. When $A\lesssim B$ and $B\lesssim A$, we write $A\backsimeq B$.

 \section{Boundedness of $h_b$: $\mathcal H^{\Phi_1}(\mathbb
 B^n)\rightarrow \mathcal H^{\Phi_2}(\mathbb
 B^n)$;  $(\Phi_1,\Phi_2)\in \mathscr{L}_p \times \mathscr{U}^q$ }
 The section is devoted to the proof of Theorem \ref{maintheorem1} and Theorem \ref{maintheorem2}.
\subsection{Some properties of growth functions}
We collect in this subsection few properties of growth functions we shall use later.

We start with this useful proposition which gives a relation between functions in the classes $\mathscr{L}_p$ and $\mathscr{U}^q.$
\begin{prop}\label{phiandinverse}
The following assertion holds:\\
\begin{center}
    $\Phi\in \mathscr{L}_p$ if and only if $\Phi^{-1}\in \mathscr{U}^{1/p}.$
\end{center}
\end{prop}
\begin{proof}
It is clear that $\Phi$ is a growth function if and only if $\Phi^{-1}$ is a growth function. Also, it is clear that $\frac{\Phi(t)}{t}$ is non-increasing if and only if  $\frac{\Phi^{-1}(t)}{t}$ is non-decreasing. Hence what is left to prove is that $\Phi$ is of lower type $p$ if and only if $\Phi^{-1}$ is of upper type $1/p$.

Suppose there exists $C>1$ so that for every $s\leq 1$ and all $t>0$, we have
\begin{eqnarray}\label{lowerp}
\Phi(st)\leq Cs^p\Phi(t).
\end{eqnarray}
Let $x\geq 1$ and $y>0$. Applying inequality (\ref{lowerp}) to
 $s=\frac 1{(Cx)^{1/p}}\;\textrm{and}\; t=\Phi^{-1}(y) (Cx)^{1/p},$ we obtain

\begin{eqnarray*}
y\leq \frac 1{x}\Phi\left(\Phi^{-1}(y) (Cx)^{1/p}\right),
\end{eqnarray*}
and consequently, $\Phi^{-1}(xy)\leq Ax^{1/p}\Phi^{-1}(y).$ Hence $\Phi^{-1}$ is of upper type $1/p$. The arguments could be reversed. This ends the proof of the proposition.
\end{proof}

\begin{lem}\label{stabilityoflowertypeclass}
Let $\Phi_1\in \mathscr{L}_p$ and $\Phi_2\in \mathscr{U}^q$, and $\Psi_2$ the complementary function of $\Phi_2.$ Let $\Phi$ be such that
\begin{equation*}
\Phi^{-1}(t):=\Phi_1^{-1}(t)\Psi_2^{-1}(t).
\end{equation*}
Then $\Phi\in \mathscr{L}_r$ for some $r\leq p.$
\end{lem}

\begin{proof}
Let us write $\frac{\Phi^{-1}(t)}{t}=\left(\frac{\Phi_1^{-1}(t)}{t}\right)\Psi_2^{-1}(t)$ and remark that $\frac{\Phi_1^{-1}(t)}{t}$ and $\Psi_2^{-1}(t)$
are non-decreasing. We deduce easily that $\frac{\Phi^{-1}(t)}{t}$ is non-decreasing.

By Proposition \ref{phiandinverse}, it just remains to prove that $\Phi^{-1}$ is of upper type $1/r$.
Let $s\geq 1$ and $t>0$, applying Proposition \ref{phiandinverse} to $\Phi_1$ and using the fact that  $\frac{\Psi_2^{-1}(t)}{t}$ is non-increasing, we obtain
\begin{eqnarray*}
\Phi^{-1}(st)&=&\Phi^{-1}_1(st)\Psi_2^{-1}(st)\\
&\lesssim& s^{1/p}\Phi^{-1}_1(t)st \frac{\Psi_2^{-1}(st)}{st}\\
&\lesssim &s^{\frac{p+1}{p}}\Phi^{-1}(t).
\end{eqnarray*}
\end{proof}

\begin{lem}
Let $\Phi_1$ be a growth function and $\Phi_2\in \mathscr{U}^q$, $\rho_i(t)=\frac{1}{t\Phi_i^{-1}(1/t)}$ and $\Psi_2$ the complementary of $\Phi_2.$ Then if
$$\rho_{\Phi}:=\frac{\rho_1}{\rho_2},$$
we also have
\begin{equation*}
\Phi^{-1}(t)\simeq \Phi_1^{-1}(t)\Psi_2^{-1}(t)
\end{equation*}
and vice-versa.
\end{lem}
\begin{proof}
It is enough to prove that
\begin{equation}\label{relationcomplementaryandrho}
\rho_2(t)\simeq \Psi_2^{-1}(1/t).
\end{equation}

This follows easily from the fact that
\begin{equation}\label{relationfunctionandcomplementary}
t\leq \Phi_2^{-1}(t)\Psi_2^{-1}(t)\leq 2t.
\end{equation}
\end{proof}

\begin{lem}\label{decreasing}
Let $\Phi\in \mathscr{U}^q.$ Then for  $0<m<\infty$, $\Phi_m(t):=\Phi(t^{1/m})$ is of lower type $1/m$ and if moreover $m$ is
large enough, then $\frac{\Phi(t)}{t^m}$ is non-increasing.
\end{lem}

\begin{proof}
Clearly, $\Phi_m$ is of lower type $1/m$ since $\Phi$ is of lower type $1$.
Let us recall that there are constants $c_1>0$ and $c_2>0$ such that for any $t>0$, $$c_1\frac{\Phi(t)}{t}\le \Phi'(t)\le c_2\frac{\Phi(t)}{t}.$$
Now, put $g(t)=\frac{\Phi(t)}{t^m}$ with $m\ge c_2$. We obtain easily that
\begin{eqnarray*}
g'(t) &=& \frac{t\Phi'(t)-m\Phi(t)}{t^{m+1}}\\ &\le& \frac{(c_2-m)\Phi(t)}{t^{m+1}}\\ &\le& 0.
\end{eqnarray*}
\end{proof}

\begin{lem}\label{conditionforlowertype}
Let $\Phi_1$ and $\Phi_2$ be in $ \mathscr{U}^q$, and $\Psi_2$ the complementary function of $\Phi_2$.  Let $\Phi$ be such that $\Phi^{-1}(t)=\Phi_1^{-1}(t)\Psi_2^{-1}(t)$. We suppose that $\Phi_2$ satisfies the Dini condition (\ref{dinicondition}) and that $$\frac{\Phi_2^{-1}\circ\Phi_1(t)}{t}\quad\textrm{is non-increasing}.$$
Then $\Phi\in \mathscr{L}_p$ for some $p>0$.
\end{lem}
\begin{proof}
Using (\ref{relationfunctionandcomplementary}) and Proposition \ref{phiandinverse}, it is enough to prove that $$\phi^{-1}(t)=\frac{t\Phi_1^{-1}(t)}{\Phi_2^{-1}(t)}\; \textrm{is in }\; \mathscr{U}^{1/p}.$$
We have $\phi^{-1}(\Phi_1(t))=\frac{t\Phi_1(t)}{\Phi_2^{-1}\circ\Phi_1(t)}$, so that $\phi^{-1}$ is a growth function such that  $\frac{\phi^{-1}(t)}{t}$ is non-decreasing. It is left to prove that $\phi^{-1}$ is of upper type.

Let $s\geq 1$ and $t>0$. Using  the fact   $\frac{\Psi_2^{-1}(t)}{t}$ and   $\frac{\Phi_1^{-1}(t)}{t}$ are non-increasing, we obtain
\begin{eqnarray*}
\phi^{-1}(st)&\simeq &\Phi^{-1}_1(st)\Psi_2^{-1}(st)\\
&\lesssim& s \Phi^{-1}_1(t)s \Psi_2^{-1}(s)\\
&\lesssim &s^{2}\phi^{-1}(t).
\end{eqnarray*}
Hence $\phi^{-1}\in \mathscr{U}^{2}$. The proof is complete.
\end{proof}

\subsection{Proof of Theorem \ref{maintheorem1} and Theorem \ref{maintheorem2}}
We recall the following definition of a molecule (see \cite{BG} and the references therein).
\begin{defn}
A holomorphic function $A\in \mathcal H^2(\mathbb B^n)$ is called a molecule of order $L$, associated to the ball $B:=B(z_0,r_0)\subset \mathbb S^n$, if it satisfies
\begin{equation}\label{molecule}
||A||_{mol(B,L)}:=\left(\sup_{r<1}\int_{\mathbb S^n}\left(1+\frac{d(z_0,r\xi)^{L+n}}{r_0^{L+n}}\right)|A(r\xi)|^2\frac{d\sigma(\xi)}{\sigma(B)}\right)^{1/2}<\infty.
\end{equation}
\end{defn}
We have used the notation $d(z,w):=|1-z\overline{w}|$ for $z,w\in\overline{\mathbb B^n}$.

It is proved in \cite{BG} that for $\Phi\in \mathscr{L}_p$, every molecule $A$ of order $L$ so that $L>L_p:=2n(1/p -1)$ belongs to $\mathcal H^{\Phi}(\mathbb B^n)$ with
\begin{equation}\label{moleculeinHardy-orlicz}
||A||_{\mathcal H^{\Phi}}\lesssim \Phi(||A||_{mol(B,L)})\sigma(B).
\end{equation}
The following molecular decomposition for functions in some Hardy-Orlicz spaces is proved in \cite{BG}.
\begin{thm}
Let $\Phi\in \mathscr{L}_p$. For any $f\in \mathcal H^{\Phi}(\mathbb B^n)$, there exist molecules $A_j$ of order $L>L_p$, associated to the balls $B_j$, so that $f$ may be written as
$$f=\sum_j A_j$$
with  $||f||_{\mathcal H^{\Phi}}\simeq \sum_j \Phi(||A_j||_{mol(B_j,L)})\sigma(B_j).$
\end{thm}

We have  the following generalization of \cite[Lemma 3.9]{LLQR}.
\begin{lem}\label{Orlicznormofboundedfunction}
Let $(\Omega, \mathbb P)$ be a probability space, $\Phi$ a convex function such that $\Phi(0)=0$, and $0<p<\infty$. Then for every $g\in L^\infty (\Omega)$,
$$||g||^{lux}_{\mathcal H^{\Phi^p}}\le \frac{||g||_{\infty}}{\Phi^{-1}(||g||_{\infty}/||g||_{L^p})}.$$
\end{lem}
\begin{proof}
The proof follows exactly as in \cite{LLQR}. We give it here for completeness. We may assume that $||g||_\infty=1$. Since $\Phi$ is convex with $\Phi(0)=0$ and $|g|\le 1$,
we obtain for every $C>0$, $$\int_{\Omega}\Phi^p(\frac{|g|}{C})d\mathbb P\le \int_{\Omega}|g|^p\Phi^p(\frac{1}{C})d\mathbb P=||g||_{L^p}^p\Phi^p(\frac{1}{C}).$$
Now one sees that $||g||_{L^p}^p\Phi^p(\frac{1}{C})\le 1$ if and only if $C\ge 1/\Phi^{-1}(1/||g||_{L^p})$, and from this follows the proof of the lemma.
\end{proof}

The following is a direct consequence of the above lemma.
\begin{lem}\label{anorliczfunction2}
Suppose that $\Phi\in \mathscr{U}^q$ and let $\Psi$ be its complementary. For  each $a\in\mathbb B^n$ and $0<p<\infty$, let
$$g(z)=\Psi^{-1}\left(\frac{1}{(1-|a|)^{n/p}}\right)\left(\frac{1-|a|}{1-\langle z,a\rangle}\right)^{\frac{2n}{p}}.$$ We have $||g||^{lux}_{\mathcal H^{\Psi^p}}\lesssim 1.$
\end{lem}
We now describe a factorization of each molecule. This is the main ingredient in the proof of our  results in this section.
\begin{thm}\label{factorization}
Let $\Phi_1\in \mathscr{L}_p$ and $\Phi_2\in \mathscr{U}^q$, $\rho_i(t)=\frac{1}{t\Phi_i^{-1}(1/t)}$ and $\Psi_2$ the complementary of $\Phi_2.$ Let $\Phi$ be such that $\rho(t)=\rho_\Phi(t):=\frac{\rho_1}{\rho_2}.$ We assume moreover that $\Phi_2$ satisfies the Dini condition (\ref{dinicondition}). Then a molecule $A$ associated to the ball $B$ may be written as $fg$, where $f$ is  a molecule and $g\in \mathcal H^{\Psi_2}(\mathbb B^n)$. Moreover, for $L, L^\prime$ given with $L^\prime + 4n \leq L,$  $f$ and $g$ may be chosen such that
\begin{equation}\label{estimate2}
||g||_{\mathcal H^{\Psi_2}}\lesssim 1, \,\,\,\textrm{and}\,\,\, ||f||_{mol(B,L^\prime)}\lesssim \frac{||A||_{mol(B,L)}}{\rho_2(\sigma(B))},
\end{equation}
or such that
\begin{equation}\label{estimate3}
||g||_{\mathcal H^{\Psi_2}}\lesssim 1, \,\,\,\textrm{and}\,\,\,  ||f||^{lux}_{\mathcal H^{\Phi_1}}\lesssim ||A||_{mol(B,L)}\sigma(B)\rho(\sigma(B)).
\end{equation}
\end{thm}
\begin{proof}
Suppose $B:=B(z_0,r)\subset \mathbb S^n,$ with $z_0\in \mathbb S^n$ and $r<1$. Let $a=(1-r)z_0$ and take
$$g(z)=\Psi_2^{-1}\left(\frac{1}{(1-|a|)^{n}}\right)\left(\frac{1-|a|}{1-\langle z,a\rangle}\right)^{2n}.$$ By Lemma \ref{anorliczfunction2}, we have $||g||_{\mathcal H^{\Psi_2}} \lesssim 1$, and using (\ref{relationcomplementaryandrho}) we also have that
\begin{equation}\label{estimate4}
|g(z)|\simeq \rho_2(\sigma(B)) \frac{r^{2n}}{|1-z\overline{a}|^{2n}}.
\end{equation}
We know that $|1-z\overline{a}|^{2n}\lesssim (d(z,z_0)+r)^{2n}$. Hence
$$|g(z)|\gtrsim \rho_2(\sigma(B))\frac{r^{2n}}{ (d(z,z_0)+r)^{2n}}.$$
Having this and under the condition on $L$ and $L^\prime$, the rest of the proof follows as in the proof of the analogue result in  \cite[Theorem 4.3]{BS}, we omit these details.
\end{proof}

Having Theorem \ref{factorization} and the techniques in \cite{BS}, the proof of Theorem \ref{maintheorem1} and Theorem \ref{maintheorem2} is now routine. Indeed, the sufficient part of Theorem \ref{maintheorem1} is an application of Proposition \ref{volbergtolokonnikov} and the sufficiency in Theorem \ref{maintheorem2} follows from   Theorem \ref{duality2}, Theorem \ref{duality1} and Proposition \ref{volbergtolokonnikov}. The necessity part follows as in the proof of Theorem 1.8 and 1.9 in \cite{BS} using Theorem \ref{factorization} instead of  Theorem 4.3 in \cite{BS}. We omit the details. This ends the proof of Theorem \ref{maintheorem1} and Theorem \ref{maintheorem2}.

\section{Boundedness of $h_b$: $\mathcal H^{\Phi_1}(\mathbb
 B^n)\rightarrow \mathcal H^{\Phi_2}(\mathbb
 B^n)$;  $\Phi_i\in \mathscr{U}^q,\; i=1,2.$}\label{section2}
This section is devoted to the study of the boundedness of the Hankel operator $h_b$, between  two Hardy-Orlicz spaces $\mathcal H^{\Phi_1}(\mathbb
 B^n)$ and $\mathcal H^{\Phi_2}(\mathbb
 B^n)$ with $\Phi_i \in \mathscr{U}^q;$ $i=1,2.$ The main tools we need are in \cite{BG} where atomic and molecular decomposition for functions in $\mathcal H^{\Phi}(\mathbb B^n)$ with $\Phi\in \mathscr{L}_p$ are described.  But since we are dealing here with convex functions, we will need to consider some simple generalizations of those results in order to get rid of the present difficulty. These extensions are explained in the next subsection and in many cases the proofs just follow the lines of \cite{BG} where we will refer for further details.
\subsection{Generalization of atomic and molecular decomposition }
For $\Phi\in \mathscr{L}_p$, the atomic decomposition and molecular decomposition for $\mathcal H^{\Phi}(\mathbb B^n)$ are described using square integrable functions \cite{BG}. One classical result used  in various arguments is the fact that the Szeg\"o projection is bounded on $L^2(\mathbb S^n)$. It is well known that, for all $1<m<\infty,$ the Szeg\"o projection is bounded from $L^m(\mathbb S^n)$ to $\mathcal H^{m}(\mathbb B^n)$. This fact allows us to obtain atomic and molecular decompositions for $\mathcal H^{\Phi}(\mathbb B^n)$ using $m-$integrable functions.

We will now give a precise description of what we are talking about. In the sequel, $m>1$ will be a fixed real.
\begin{defn}\label{matom}
A function a in $L^m(\mathbb S^n)$ is called an $m-$atom of order $N\in \mathbb N$ associated to the ball $B:=B(z_0,r_0)$, for some $z_0\in \mathbb S^n$, if $supp\; a\subset B$ and when $r_0<\delta, $
$$\int_{\mathbb S^n}a(\xi)P(\xi)d\sigma(\xi)=0 \;\textrm{for every}\; P\in \mathcal{P}_N(z_0).$$
\end{defn}
We obtain the following atomic decomposition.
\begin{thm}
Let $N\in\mathbb N$ be larger than $N_{p,m}:= mn(1/p-1)-1.$ Given any $f\in \mathcal H^{\Phi}(\mathbb B^n)$, there exist $m-$atoms $a_j$ of order $N$ such that (in the distribution sense)
$$f=P\left(\sum_{j=0}^\infty a_j\right)=\sum_{j=0}^\infty P(a_j).$$
Moreover,
$$\sum_{j=0}^\infty \sigma(B_j)\Phi(||a_j||_m\sigma(B_j)^{-1/m})\simeq ||f||_{\mathcal H^{\Phi}(\mathbb B^n)}.$$
\end{thm}

\begin{defn}
A holomorphic function $A\in \mathcal H^m(\mathbb B^n)$ is called a $m-$molecule of order $L$, associated to the ball $B:=B(z_0,r_0)\subset \mathbb S^n$, if it satisfies
\begin{equation}\label{molecule2}
||A||_{mol(B,L,m)}:=\left(\sup_{r<1}\int_{\mathbb S^n}\left(1+\frac{d(z_0,r\xi)^{L+n}}{r_0^{L+n}}\right)|A(r\xi)|^m\frac{d\sigma(\xi)}{\sigma(B)}\right)^{1/m}<\infty.
\end{equation}
\end{defn}
The following proposition replaces Proposition $1.9$ in \cite{BG}. The proof is similar.
\begin{prop}\label{projectionofmatom}
For an $m-$atom a of order $N$ associated to a ball $B\subset \mathbb S^n$, its Szeg\"o projection $P(a)$ is a $m-$molecule associated to the ball $\tilde B$ of double radius, of any order $L< (m-2)n + \frac{N+1}{2}m$. It satisfies
\begin{equation}\label{projectionofatom}
||A||_{mol(\tilde B,L,m)} \lesssim ||a||_m\sigma(B)^{-1/m}.
\end{equation}
\end{prop}
This yields the following molecular decomposition.
\begin{thm}\label{moleculardecomposition2}
Let $\Phi\in \mathscr{L}_p$. For any $f\in \mathcal H^{\Phi}(\mathbb B^n)$, there exist $m-$molecules $A_j$ of order $L>L_{p,m}:=m n(1/p-1)$, associated to the balls $B_j$, so that $f$ may be written as
$$f=\sum_j A_j$$
with  $||f||_{\mathcal H^{\Phi}}\simeq \sum_j \Phi(||A_j||_{mol(B_j,L,m)})\sigma(B_j).$
\end{thm}

Let $\rho$ be a weight, we  define a family of weighted BMOA    spaces $BMOA(\rho,m)$.
\Be\label{bmoaweight2}
BMOA(\rho,m):=\left\lbrace f\in \mathcal H^{m}(\mathbb B^n):\quad  \sup_{r<1} ||f_r||_{BMO(\rho,m)}<\infty\right\rbrace, \Ee
where
$$||f||_{BMO(\rho,m)}=\left(\sup_B\inf_{R\in \mathcal P_N(B) }\frac 1{(\varrho(\sigma(B)))^m\sigma(B)}\int_B|f -R|^m d\sigma\right)^{1/m}.$$

The work of Viviani \cite{V}  allows us to see that $BMOA(\rho,m)$ are all the same as $m\geq 1$. Hence the dual space $(\mathcal H^{\Phi}(\mathbb B^n))^*$ of $\mathcal H^{\Phi}(\mathbb B^n)$ can be identified with $BMOA(\rho,m),$ for each fixed $m$.

 The following proposition will be used to replace the inequality (\ref{moleculeinHardy-orlicz}) in the case of Orlicz function $\Phi\in \mathscr{U}^q.$

\begin{prop}
Let $\Phi\in \mathscr{U}^q$ and $0<p\le 1$. Then for $m$ large enough, any $m-$molecule A of order $L$ associated to a ball $B$, such that $L>\frac{nm}{p} -2n$, belongs to $\mathcal H^{\Phi^p}(\mathbb B^n)$ with
\begin{equation}\label{moleculeinHardy-orlicz3}
||A||_{\mathcal H^{\Phi^p}}\lesssim \Phi^p(||A||_{mol(B,L,m)})\sigma(B).
\end{equation}
\end{prop}
\begin{proof}
We first remark that by Lemma \ref{decreasing}, $\Phi_m=\Phi(t^{1/m})\in\mathscr{L}_{1/m} $. Consequently, we may suppose that $\Phi_m$ is concave, which implies that $\Phi_m^p$ is also concave. The proof then follows as in the proof of Proposition $1.10$ in \cite{BG}, using often the Jensen Inequality in the following way:
\begin{equation}\label{jensen}
\int_X\Phi^p(f)d\mu = \int_X\Phi_m^p(f^m)d\mu\leq \Phi_m^p\left( \int_X f^m d\mu\right)=\Phi^p\left(||f||_{L^m(X,d\mu)}\right).
\end{equation}
Here, $d\mu$ is a probability measure and $f$ is a positive function on the measure space $(X,d\mu).$
\end{proof}

\subsection{Boundedness of $h_b$: $\mathcal H^{\Phi_1}(\mathbb B^n)\mathcal\rightarrow H^{\Phi_2}(\mathbb B^n)$}
We are now ready to give our result for the boundedness of the Hankel operators from
$ \mathcal H^{\Phi_1}(\mathbb B^n)$ into $ \mathcal H^{\Phi_2}(\mathbb B^n)$ in the case where $\Phi_1$ and $\Phi_2$ are both convex.  As in the previous section, we will follow the techniques in \cite{BS}. The generalization of molecular decomposition described in the previous subsection will be used in the present situation.

 We first describe a factorization of each $m-$molecule as a product of functions in  Hardy-Orlicz spaces related to power of convex growth functions.

\begin{thm}\label{factorization2gene}
Let $\Phi_1$ and $\Phi_2\in \mathscr{U}^q$, $\rho_i(t)=\frac{1}{t\Phi_i^{-1}(1/t)}$, $0<p\le 1$, and $\Psi_2$
the complementary function of $\Phi_2.$ Let $\Phi$ be such that $\rho(t)=\rho_\Phi(t):=\frac{\rho_1}{\rho_2}.$
 Then for $m$ large enough, a $m-$molecule $A$ associated to the ball $B$ may be written as $fg$, where $f$ is
 an $m-$molecule and $g\in \mathcal H^{\Psi_2^p}(\mathbb B^n)$. Moreover, for $L, L^\prime$ given with $L^\prime + \frac{2mn}{p} \leq L,$ $f$ and $g$ may be chosen such that
\begin{equation}\label{estimateaa}
||g||_{\mathcal H^{\Psi_2^p}}\lesssim 1, \,\,\,\textrm{and}\,\,\, ||f||_{mol(B,L^\prime,m)}\lesssim \frac{||A||_{mol(B,L,m)}}{\rho_2(\sigma(B)^{1/p})},
\end{equation}
or such that
\begin{equation}\label{estimateab}
||g||_{\mathcal H^{\Psi_2^p}}\lesssim 1, \,\,\,\textrm{and}\,\,\,  ||f||^{lux}_{\mathcal H^{\Phi_1^p}}\lesssim ||A||_{mol(B,L,m)}\sigma(B)^{1/p}\rho(\sigma(B)^{1/p}).
\end{equation}
\end{thm}
\begin{proof}

Suppose $B:=B(z_0,r)\subset \mathbb S^n,$ with $z_0\in \mathbb S^n$ and $r<1$. Let $a=(1-r)z_0$  and take
$$g(z)=\Psi_2^{-1}\left(\frac{1}{(1-|a|)^{n/p}}\right)\left(\frac{1-|a|}{1-\langle z,a\rangle}\right)^{2n/p}.$$  Hence by Lemma \ref{anorliczfunction2}, we have $||g||_{\mathcal H^{\Psi^p_2}} \lesssim 1$, and using (\ref{relationcomplementaryandrho}) we also have that
\begin{equation}\label{estimateac}
|g(z)|\gtrsim \rho_2(\sigma(B)^{1/p})\frac{r^{\frac{2n}{p}}}{ (d(z,z_0)+r)^{\frac{2n}{p}}},\quad z\in \overline{\mathbb B^n}.
\end{equation}
We set $f=A/g$ and $\tilde B=B(z_0,2r)$ and proceed to prove the second part of (\ref{estimateaa}). We obtain
\begin{eqnarray*}
||f||^m_{mol(B,L^\prime,m)}&=& \int_{\mathbb S^n}\left(1+\frac{d(z_0,\xi)^{L^\prime +n}}{r^{L^\prime +n}}\right)|f(\xi)|^m\frac{d\sigma(\xi)}{\sigma(B)}\\
&\simeq& \int_{\tilde B}|f(\xi)|^m\frac{d\sigma(\xi)}{\sigma(B)} + \int_{\mathbb S^n\backslash\tilde B}\left(\frac{d(z_0,\xi)}{r}\right)^{L^\prime +n}|f(\xi)|^m\frac{d\sigma(\xi)}{\sigma(B)}\\
&=&I +II,
\end{eqnarray*}
where
$$I=\int_{\tilde B}|f(\xi)|^m\frac{d\sigma(\xi)}{\sigma(B)} \quad\textrm{and}\quad
II =\int_{\mathbb S^n\backslash\tilde B}\left(\frac{d(z_0,\xi)}{r}\right)^{L^\prime +n}|f(\xi)|^m\frac{d\sigma(\xi)}{\sigma(B)}$$
In $ \tilde B$, we have $|g(\xi)|\gtrsim \rho_2(\sigma(B)^{1/p})$  so that
\begin{eqnarray*}
I &\lesssim & \frac 1{\left(\rho_2(\sigma(B)^{1/p})\right)^m}\int_{\tilde B}|A(\xi)|^m\frac{d\sigma(\xi)}{\sigma(B)} \lesssim  \frac {||A||^m_{mol(B,L,m)}}{\left(\rho_2(\sigma(B)^{1/p})\right)^m}.
\end{eqnarray*}
In  $\mathbb S^n\backslash\tilde B$, we have $|g(\xi)|\gtrsim \rho_2(\sigma(B)^{1/p})\frac{r^{2n/p}}{d(z,z_0)^{2n/p}}$ so that
\begin{eqnarray*}
II &=& \int_{\mathbb S^n\backslash\tilde B}\left(\frac{d(z_0,\xi)}{r}\right)^{L^\prime +n}|f(\xi)|^m\frac{d\sigma(\xi)}{\sigma(B)}\\
&\lesssim & \frac 1{\left(\rho_2(\sigma(B)^{1/p})\right)^m}\int_{\mathbb S^n\backslash\tilde B}\left(\frac{d(z_0,\xi)}{r}\right)^{L^\prime +n +\frac{2nm}{p}}|A(\xi)|^m\frac{d\sigma(\xi)}{\sigma(B)}\\
&\lesssim & \frac {||A||^m_{mol(B,L,m)}}{\left(\rho_2(\sigma(B)^{1/p})\right)^m}.
\end{eqnarray*}
This proves (\ref{estimateaa}). It remains to prove (\ref{estimateab}).

By homogeneity, it is sufficient to prove that for $||A||_{mol(B,L,m)}\sigma(B)^{1/p}\rho(\sigma(B)^{1/p})=1$, the function that has been chosen is such that $\int_{\mathbb S^n}\Phi_1^p(|f|)d\sigma\lesssim 1.$ By  (\ref{moleculeinHardy-orlicz3}), it is enough to prove that $||f||_{mol(B,L^\prime,m)})\lesssim \Phi_1^{-1}\left(\frac 1{\sigma(B)^{1/p}}\right)=\frac 1{\sigma(B)^{1/p}\rho_1(\sigma(B)^{1/p})}.$ This holds by (\ref{estimateaa}) and the definition of $\rho$.
\end{proof}

Taking $\Phi_1=\Phi_2$, we obtain in particular the following.
\begin{cor}\label{factorization3gene}
Let $\Phi\in \mathscr{U}^q$, $0<p\le 1$, $\rho(t)=\frac{1}{t\Phi^{-1}(1/t)}$, and $\Psi$ the complementary function of $\Phi.$  Then for $m$ large enough, a $m-$molecule $A$ associated to the ball $B$ may be written as $fg$, where $f$ is  an $m-$molecule and $g\in \mathcal H^{\Psi^p}(\mathbb B^n)$. Moreover, for $L, L^\prime$ given with $L^\prime + \frac{2mn}{p} \leq L,$ $f$ and $g$ may be chosen such that
\begin{equation}\label{estimate23}
||g||_{\mathcal H^{\Psi^p}}\lesssim 1, \,\,\,\textrm{and}\,\,\, ||f||_{mol(B,L^\prime,m)}\lesssim \frac{||A||_{mol(B,L,m)}}{\rho(\sigma(B)^{1/p})},
\end{equation}
or such that
\begin{equation}\label{estimate33}
||g||_{\mathcal H^{\Psi^p}}\lesssim 1, \,\,\,\textrm{and}\,\,\,  ||f||^{lux}_{\mathcal H^{\Phi^p}}\lesssim ||A||_{mol(B,L,m)}\sigma(B)^{1/p}.
\end{equation}
\end{cor}

This leads to the following generalization of the weak factorization of the Hardy spaces $\mathcal H^p(\mathbb B^n)$, $0<p\le 1$, obtained in \cite{CRW, GP}
\begin{thm}\label{weakfactoHp}
Let $0<p\le 1$, $\Phi\in \mathscr{U}^q$, $\rho(t)=\frac{1}{t\Phi^{-1}(1/t)}$, and $\Psi$ the complementary function of $\Phi.$ Then given any
$f\in \mathcal H^p(\mathbb B^n)$ there exist $f_j\in \mathcal H^{\Phi^p}(\mathbb B^n)$, $g_j\in \mathcal H^{\Psi^p}(\mathbb B^n)$, $j\in \mathbb N$
 such that $$f=\sum_{j=0}^\infty f_jg_j$$
 and
\begin{equation}\label{hpnormineq}
\sum_{j}||g_j||^{lux}_{\mathcal H^{\Psi^p}}||f_j||^{lux}_{\mathcal H^{\Phi^p}}\lesssim ||f||_{p}.
\end{equation}
\end{thm}
\begin{proof}
%Let us recall that $t^{1/p}\backsimeq \Phi^{-1}(t^{1/p})\Psi^{-1}(t^{1/p})$.
 Applying Theorem \ref{moleculardecomposition2} to $\Phi(t)=t^p$, we know that there exist $m-$molecules $A_j$ of order $L>L_{p,m}:=m n(1/p-1)$, associated to the balls $B_j$, so that $f$ may be written as
$$f=\sum_j A_j$$
with  $||f||_{\mathcal H^{p}}\simeq \sum_j ||A_j||^p_{mol(B_j,L,m)}\sigma(B_j).$

The weak factorization then follows from the factorization of each molecule
as obtained in Corollary \ref{factorization3gene} with
\begin{equation*}
||g_j||_{\mathcal H^{\Psi^p}}\lesssim 1, \,\,\,\textrm{and}\,\,\,  ||f_j||^{lux}_{\mathcal H^{\Phi^p}}\lesssim ||A_j||_{mol(B_j,L,m)}\sigma(B_j)^{1/p}.
\end{equation*}
The inequality (\ref{hpnormineq}) follows easily from the above observation:
\begin{eqnarray*}
\sum_{j}||g_j||^{lux}_{\mathcal H^{\Psi^p}}||f_j||^{lux}_{\mathcal H^{\Phi^p}} &\lesssim& \sum_{j}||f_j||^{lux}_{\mathcal H^{\Phi^p}}\\ &\lesssim&
\sum_{j}||A_j||_{mol(B_j,L,m)}\sigma(B_j)^{1/p}\\ &\lesssim& \left(\sum_{j}||A_j||_{mol(B_j,L,m)}^p\sigma(B_j)\right)^{1/p}\\ &\lesssim& ||f||_{p}.
\end{eqnarray*}
The proof is complete.
\end{proof}

Using Theorem \ref{factorization2gene}, we are ready to prove the following result about boundedness of the Hankel operator $h_b$ in the case where the growth functions are both convex.

\begin{thm}\label{maintheorem3}
Let $\Phi_1$ and $\Phi_2$ in $ \mathscr{U}^q$, and $\rho_i(t)=\frac{1}{t\Phi_i^{-1}(1/t)}$. \\
We suppose that:
\begin{itemize}
\item[(i)] $\Phi_2$ satisfies the Dini condition (\ref{dinicondition})
\item[(ii)]  $\frac{\Phi_1^{-1}(t)\Psi_2^{-1}(t)}{t}$ is non-decreasing or $\Phi_1=\Phi_2.$
\end{itemize}
 Then   the Hankel operator $h_b$ extends into a bounded operator from $\mathcal H^{\Phi_1}(\mathbb
 B^n)$ into $\mathcal H^{\Phi_2}(\mathbb
 B^n)$ if and only if its symbol $b$ belongs to  $ BMOA(\rho),$ where
\begin{equation*}
\rho=\rho_{\Phi}:=\frac{\rho_1}{\rho_2}.
\end{equation*}
\end{thm}
\begin{proof}
We first remark that using Propostion \ref{phiandinverse}, we have that $(ii)$  implies that $\Phi\in \mathscr{L}_p$ for some $p$ so that $(\mathcal H^{\Phi}(\mathbb B^n))^*= BMOA(\rho_\Phi)$, and we have the factorization of $m-$molecules. The sufficient part follows from   Theorem \ref{duality2}, Theorem \ref{duality1} and Proposition \ref{volbergtolokonnikov}. Indeed,
\begin{eqnarray*}
||h_b(f)||^{lux}_{\mathcal H^{\Phi_2}(\mathbb B^n)}&=& \sup_{||g||^{lux}_{\mathcal H^{\Psi_2}(\mathbb B^n)}=1}|\langle h_b(f),g\rangle| = \sup_{||g||^{lux}_{\mathcal H^{\Psi_2}(\mathbb B^n)}=1}|\langle b,fg\rangle|\\
&\lesssim& \sup_{||g||^{lux}_{\mathcal H^{\Psi_2}(\mathbb B^n)}=1}\left(||b||_{BMOA(\rho_\Phi)} ||fg||^{lux}_{\mathcal H^{\Phi}(\mathbb B^n)}\right)\\
&\lesssim& \sup_{||g||^{lux}_{\mathcal H^{\Psi_2}(\mathbb B^n)}=1}\left(||b||_{BMOA(\rho_\Phi)} ||f||^{lux}_{\mathcal H^{\Phi_1}(\mathbb B^n)}||g||^{lux}_{\mathcal H^{\Psi_2}(\mathbb B^n)}\right)\\
&\lesssim&||b||_{BMOA(\rho_\Phi)} ||f||^{lux}_{\mathcal H^{\Phi_1}(\mathbb B^n)}.
\end{eqnarray*}
Now, we assume that $h_b$ is bounded from $\mathcal H^{\Phi_1}(\mathbb
 B^n)$ into $\mathcal H^{\Phi_2}(\mathbb
 B^n)$ and prove that $b$ belongs to $BMOA(\rho)$. Since $BMOA(\rho)=BMOA(\rho, m)$ for any $m\geq 1$, it is enough to prove that there exist  constants $C$ and $m^\prime$ such that, for each ball $B$, we can find a polynomial $R\in \mathcal P_N(B)$ such that

\begin{equation}\label{BMOAinequality}
\left(\int_B|b-R|^{m^\prime}d\sigma\right)^{1/m^\prime}\leq C\left(\sigma (B)\right)^{1/m^\prime}\rho(\sigma(B)).
\end{equation}
Let $B=B(z_0,r)$ be a ball in $\mathbb S^n$, we take for $R$ the orthogonal projection of $b$ onto $\mathcal P_N(B),$ and let $a:=\chi_B(b-R)$ so that $a$ is an $m^\prime-$ atom.

 For $l\in L^m(\mathbb S^n, d\sigma)$ we will denote by $R(l)$ the orthogonal projection of $l$ onto $\mathcal P_N(B),$ so that the function $a_l=\chi_B(l-R(l))$ is an $m-$atom (see Definition \ref{matom}) associated to $B$.  We claim that there exists an absolute constant $C=C(N,m,n)$ such that

\begin{equation}\label{estimate5}
||\chi_B(l-R(l))||_m\le C ||\chi_Bl||_m.
\end{equation}

Assume that (\ref{estimate5}) holds.  From Proposition \ref{projectionofmatom}, one knows that $A_l:=P(a_l)$ is a $m-$molecule
associated to $\tilde B$, with $||A_l||_{mol( B,L,m)}\lesssim ||a_l||_m\sigma(B)^{-1/m}$. From Theorem \ref{factorization2gene} we know that $A_l$ may be written as $f_lg_l$, with
\begin{equation*}
||g_l||_{\mathcal H^{\Psi_2}}\lesssim 1, \,\,\,\textrm{and}\,\,\,  ||f_l||^{lux}_{\mathcal H^{\Phi_1}}\lesssim ||A_l||_{mol(B,L,m)}\sigma(B)\rho(\sigma(B)).
\end{equation*}
From this, we obtain using (\ref{estimate5}), that
\begin{eqnarray*}
||a||_{m^\prime}&=& \sup_{||\chi_Bl||_m=1}|\langle a,l\rangle| =
\sup_{||\chi_Bl||_m=1}|\langle a,l-R(l)\rangle|\\
&=& \sup_{||\chi_Bl||_m=1}|\langle b,\chi_B(l-R(l))\rangle|= \sup_{||\chi_Bl||_m=1}|\langle b,a_l\rangle|\\
&=&\sup_{||\chi_Bl||_m=1}|\langle b,P(a_l)\rangle|=\sup_{||\chi_Bl||_m=1}|\langle h_b(f_l),g_l\rangle|\\
&\lesssim&\sup_{||\chi_Bl||_m=1}\left(||h_b||||f_l||^{lux}_{\mathcal H^{\Phi_1}}||g_l||^{lux}_{\mathcal H^{\Psi_2}}\right)\\
&\lesssim&  \sup_{||\chi_Bl||_m=1}\left(||h_b||||a_l||_m\sigma(B)^{1/m^\prime}\rho(\sigma(B))\right)\\
&\lesssim& ||h_b||\left(\sigma (B)\right)^{1/m^\prime}\rho(\sigma(B)).
\end{eqnarray*}

It remains to prove (\ref{estimate5}). It is clear that what we have to prove is that
\begin{equation}\label{estimate7}
||R(l)||_m\le C ||\chi_Bl||_m.
\end{equation}

Without loss of generality we can assume that $z_0=(1,0,\cdots , 0)$, so that the coordinates related to $z_0$ may be taken as the ordinary ones. Otherwise we use the action of the unitary group. In the local coordinates, the ball $B$ becomes $Q(r)=\left\lbrace  z=(t,x)\in \mathbb R^{2n-1}=\mathbb R\times \mathbb R^{2n-2}: \;|t|+|x|^2< r\right\rbrace,$ and the measure $\sigma$, the Lebesgue measure in  $\mathbb R^{2n-1}.$   In these coordinates, $ \mathcal P_N(B)$ is the space of polynomials of degree at most $N$ with support in $Q(r)$. This is a closed subspace of the Hilbert space $L^2(Q(r),dz=dtdx)$ with finite dimension $M$. %$\left(\begin{array}{c} N+2n-1\\ 2n-1 \end{array}\right). $
So if $\{P_j\}_{\{1\le j\le M\}}$ is an orthonormal basis of   $ \mathcal P_N(B)$, we have
$$R(l)=\sum_{j=1}^M \langle l,P_j\rangle P_j.$$
It follows that to prove (\ref{estimate7}), it is enough to prove that for some absolute constant $C$ (independent of $P$ and $r$),
\begin{equation}\label{estimate9}
||P||_m||P||_{m^\prime}\le C||P||^2_2 ,
\end{equation}
for any polynomial $P \in  \mathcal P_N(B) $.

The inequality (\ref{estimate9})  follows from the fact that there exist constants  $A$ and $B$ depending only on $N$ and $n$ such that, for any polynomial $P=\sum_{|\alpha|\le N }c_\alpha z^\alpha$,

\begin{equation}\label{estimate10}
A\int_{Q(1)}|P(z)|dz \leq \sum_{|\alpha|\leq N}|c_\alpha|\leq B \int_{Q(1)}|P(z)|dz.
\end{equation}

Indeed, (\ref{estimate10}) clearly shows that for any $m\ge 1,$ $$\left(\int_{Q(1)}|P(z)|^mdz\right)^{1/m}\simeq \int_{Q(1)}|P(z)|dz$$
 and the desired result then follows from the fact that $\mathcal P_N(B) $ is stable under dilations and translations.  This ends the proof of the theorem.
\end{proof}

We can observe that in the proof of the above theorem, the condition $(ii)$ is used to ensure that the resulting growth function $\Phi$ is in some $\mathcal L_p$. Hence using Lemma \ref{conditionforlowertype}, we have the following proposition.

\begin{prop}\label{maintheorem4}
Let $\Phi_1$ and $\Phi_2$ in $ \mathscr{U}^q$, and $\rho_i(t)=\frac{1}{t\Phi_i^{-1}(1/t)}$. \\
We suppose that
\begin{itemize}
\item[(i)] $\Phi_2$ satisfies the Dini condition (\ref{dinicondition})
%\item[(ii)]  for some $m\ge q, $ $\frac{\Phi_1(t)}{t^m}$ is non-increasing
\item[(ii)]  $\frac{\Phi_2^{-1}\circ\Phi_1(t)}{t}\quad\textrm{is non-increasing}$.
\end{itemize}
  Then   the Hankel operator $h_b$ extends into a bounded operator from $\mathcal H^{\Phi_1}(\mathbb
 B^n)$ into $\mathcal H^{\Phi_2}(\mathbb
 B^n)$ if and only if its symbol $b$ belongs to  $ BMOA(\rho_\Phi)= (\mathcal H^{\Phi}(\mathbb
 B^n))^*,$ where
\begin{equation*}
\rho_{\Phi}:=\frac{\rho_1}{\rho_2}.
\end{equation*}
\end{prop}

\section{Boundedness of $h_b$: $\mathcal H^{\Phi_1}(\mathbb
 B^n)\rightarrow \mathcal H^{\Phi_2}(\mathbb
 B^n)$;  $(\Phi_1,\Phi_2)\in \mathscr{U}^q\times \mathscr{L}_p $}

Let us begin this section by recalling the definition of the admissible maximal function $\mathcal {M}(f)$ of a holomorphic function $f$.  For $\xi\in \mathbb S^n$,
$$\mathcal {M}(f)(\xi)=\sup\{|f(z)|:z\in \mathbb B^n, |1-\langle \xi,z\rangle|<1-|z|^2\}.$$
 We recall that $\mathcal {H}_{weak}^1(\mathbb B^n)$ consists of functions $f\in \mathcal {H}(\mathbb B^n)$ such that,
$$\lambda \sigma\left(\{\xi\in \mathbb S^n:\mathcal {M}(f)(\xi)>\lambda\}\right)\le C\,\,\,\textrm{for any}\,\,\, \lambda>0.$$

The following result is well known.
\begin{prop}\label{h1faibleorlicz}
Let $\Phi\in \mathscr{L}_p$. Suppose that $\Phi$ satisfies the Dini's condition
\begin{equation}\label{Dinicondition1}
\int_1^\infty\frac{\Phi(t)}{t^2}dt\le C<\infty.
\end{equation}
Then $\mathcal {H}_{weak}^1(\mathbb B^n)$ embeds continuously in $\mathcal H^{\Phi}(\mathbb
 B^n)$
\end{prop}
\begin{proof}
It is enough to prove that for any $f\in \mathcal {H}_{weak}^1(\mathbb B^n)$,  $$\int_{\mathbb S^n}\Phi(\mathcal {M}(f)(\xi))d\sigma(\xi)\le C.$$
We have
\begin{eqnarray*}
\int_{\mathbb S^n}\Phi(\mathcal {M}(f)(\xi))d\sigma(\xi) &=& \int_0^\infty \sigma\left(\{\xi\in \mathbb S^n:\mathcal {M}(f)(\xi)>\lambda\}\right)\Phi'(\lambda)d\lambda\\
&=& I+J,
\end{eqnarray*}
where $$I=\int_0^1 \sigma\left(\{\xi\in \mathbb S^n:\mathcal {M}(f)(\xi)>\lambda\}\right)\Phi'(\lambda)d\lambda$$
and $$J=\int_1^\infty \sigma\left(\{\xi\in \mathbb S^n:\mathcal {M}(f)(\xi)>\lambda\}\right)\Phi'(\lambda)d\lambda.$$
Clearly,

$$I\le \sigma(\mathbb S^n)\int_0^1\Phi'(\lambda)d\lambda=C.$$
To estimate the integral $J$, we use the definition of $\mathcal {H}_{weak}^1(\mathbb B^n)$, the fact that $\Phi'(t)\backsimeq \frac{\Phi(t)}{t}$ and that $\Phi$ satisfies the Dini's condition (\ref{Dinicondition1}) to obtain
\begin{eqnarray*}
J &=& \int_1^\infty |\{\xi\in \mathbb S^n:\mathcal {M}(f)(\xi)>\lambda\}|\Phi'(\lambda)d\lambda\\ &\le&
C\int_1^\infty \frac{\Phi'(\lambda)}{\lambda}d\lambda\\ &\backsimeq& C\int_1^\infty \frac{\Phi(\lambda)}{\lambda^2}d\lambda\le C<\infty.
\end{eqnarray*}
The proof is complete.
\end{proof}

We next prove a result which generalizes the case $h_b:\mathcal H^p(\mathbb B^n)\rightarrow \mathcal H^q(\mathbb B^n)$ with $1\le p<\infty$ and $0<q<1$.
\begin{thm}\label{casavecpertesimple}
Let $\Phi_1\in \mathscr{U}^q$ and $\Phi_2\in \mathscr{L}_p$. Let $\Psi_1$ be the complementary function of $\Phi_1$ and, suppose that $\Phi_1$ satisfies the Dini's condition (\ref{dinicondition}) while $\Phi_2$
satisfies (\ref{Dinicondition1}). Then $h_b$ extends as a bounded operator from $\mathcal H^{\Phi_1}(\mathbb
 B^n)$ to $\mathcal H^{\Phi_2}(\mathbb B^n)$ if and only if $b\in \mathcal H^{\Psi_1}(\mathbb
 B^n)$.
\end{thm}
\begin{proof}
Let us begin by proving the necessity. Suppose that $h_b$ is bounded from $\mathcal H^{\Phi_1}(\mathbb
 B^n)$ to $\mathcal H^{\Phi_2}(\mathbb B^n)$. Then for any $f\in \mathcal H^{\Phi_1}(\mathbb
 B^n)$, we have $$\left|\int_{\mathbb S^n}b(\xi)\overline {f(\xi)}d\sigma(\xi)\right|=|h_bf(0)|\le C||h_b(f)||_{\mathcal H^{\Phi_2}}\le C||h_b||||f||_{\mathcal H^{\Phi_1}}.$$
  We have used the fact that $\mathcal H^{\Phi_2}(\mathbb B^n)$ is continuously contained in $\mathcal H^p(\mathbb B^n)$ (for some $p>0$), and the evaluation at $0$ is bounded on this space. It follows that $b$ belongs to the dual space of $\mathcal H^{\Phi_1}(\mathbb B^n)$
 that is $b\in \mathcal H^{\Psi_1}(\mathbb B^n)$.

 Conversely, if $b\in \mathcal H^{\Psi_1}(\mathbb B^n)$, then for any $f\in \mathcal H^{\Phi_1}(\mathbb
 B^n)$, the product $b\overline {f}$ is in $L^1(\mathbb S^n)$ by Proposition \ref{volbergtolokonnikov}. Thus, $h_b(f):=P(b\overline {f})$ is in $\mathcal {H}_{weak}^1(\mathbb B^n)$ and consequently in
 $\mathcal H^{\Phi_2}(\mathbb B^n)$ by Proposition \ref{h1faibleorlicz}. The proof is complete.
\end{proof}

\emph{Acknowledgements}: This work was done while the first author was a post-doctoral fellow at Trinity College of Dublin, a position funded by the "Irish Research Council for Science, Engineering and Technology".
%The authors would like to thank the referee for reading carefully the manuscript.
%The second author was supported by the Centre of Recerca Matem\`atica, Barcelona (Spain).

\vskip 1cm

\bibliographystyle{plain}

% ------------------------------------------------------------------------
\end{document}